\documentclass{amsart}

\usepackage{fancyhdr}
\usepackage{lastpage}

\newcommand{\bdis}{\begin{displaymath}}
\newcommand{\edis}{\end{displaymath}}
\newcommand{\be}{\begin{equation}}
\newcommand{\ee}{\end{equation}}

\newcommand{\mcal}{\mathcal}

\newtheorem{theorem}{Theorem}

\theoremstyle{definition}

\newtheorem{cor}[]{Corollary}

\theoremstyle{remark}
\newtheorem{remark}[]{Remark}

\numberwithin{equation}{section}

\begin{document}

\title{Jacob's ladders and the tangent law for short parts of the Hardy-Littlewood integral}

\author{Jan Moser}

\address{Department of Mathematical Analysis and Numerical Mathematics, Comenius University, Mlynska Dolina M105, 842 48 Bratislava, SLOVAKIA}

\email{jan.mozer@fmph.uniba.sk}

\keywords{Riemann zeta-function}

\begin{abstract}
The elementary geometric properties of the Jacob's ladders \cite{7} lead to a class of new formulae for short parts of the Hardy-Littlewood integral. This class
of formulae cannot be obtained by methods of Balasubramanian, Heath-Brown and Ivic.
\end{abstract}

\maketitle

\section{Necessity of a new expression for short parts of the Hardy-Littlewood integral}

\subsection{}

Titchmarsh-Kober-Atkinson (TKA) formula (see \cite{3}, p. 141)
\be
\int_0^\infty Z^2(t)e^{-2\delta t}{\rm
d}t=\frac{c-\ln(4\pi\delta)}{2\sin(\delta)}+\sum_{n=0}^N
c_n\delta^n+\mcal{O}(\delta^{N+1})
\ee
remained  as an isolated result for the period of 56 years. We have discovered \cite{7} the nonlinear integral equation
\be \label{1.2}
\int_0^{\mu[x(T)]}Z^2(t)e^{-\frac{2}{x(T)}t}{\rm d}t=\int_0^TZ^2(t){\rm d}t ,
\ee
in which the essence of the TKA formula is encoded. Namely, we have shown in \cite{7} that the following almost-exact expression for the Hardy-Littlewood
integral
\be \label{1.3}
\int_0^T Z^2(t){\rm d}t=\frac{\varphi(T)}{2}\ln\frac{\varphi(T)}{2}+(c-\ln(2\pi))\frac{\varphi(T)}{2}+c_0+\mcal{O}\left(\frac{\ln T}{T}\right)
\ee
takes place, where $\varphi(T)$ is the Jacob's ladder, i.e. an arbitrary solution to the nonlinear integral equation (\ref{1.2}).

\begin{remark}

Let us remind that
\begin{itemize}

\item[(A)] The Good's $\Omega$-theorem implies for the Balasubramanian's formula \cite{1}
\be \label{1.4}
\int_0^T Z^2(t){\rm d}t=T\ln T+(2c-1-\ln 2\pi)T+R(T),\quad R(T)=\mcal{O}(T^{1/3+\epsilon})
\ee
that
\bdis
\limsup_{T\to+\infty}|R(T)|=+\infty ,
\edis
i.e. the error term in (\ref{1.4}) is unbounded at $T\to\infty$.

\item[(B)] In the case of our result (\ref{1.3}) the error term tends to zero as $T$ goes to infinity, namely,
\bdis
\lim_{T\to\infty} |r(T)|=0,\quad r(T)=\mcal{O}\left(\frac{\ln T}{T}\right) ,
\edis
i.e. our formula is almost exact (see \cite{7}).

\end{itemize}

\end{remark}

\subsection{}

The Balasubramanian's formula implies, for $U_0=T^{1/3+2\epsilon}$,
\be \label{1.5}
\int_T^{T+U_0}Z^2{t}{\rm d}t=U_0\ln T+(2c-\ln 2\pi)U_0+\mcal{O}(T^{1/3+\epsilon}) .
\ee
Furthermore let us remind the Heath-Brown's estimate (see \cite{5}, (7.20) p. 178)
\be \label{1.6}
\int_{T-G}^{T+G}Z^2(t){\rm d}t=\mcal{O}\left\{ G\ln T+G\sum_K (TK)^{-1/4}\left( |S(K)|+K^{-1}\int_0^K |S(x)|{\rm d}x \right)e^{-G^2K/T} \right\}
\ee
(for the definition of used symbols see \cite{5}, (7.21)-(7.23)), uniformly for $T^\epsilon\leq G\leq T^{1/2-\epsilon}$. And finally we add the
Ivic' estimate, \cite{5}, (7.62),
\be \label{1.7}
\int_{T-G}^{T+G}Z^2(t){\rm d}t=\mcal{O}(G\ln^2 T),\quad G\geq T^{1/3-\epsilon_0},\quad \epsilon_0=\frac{1}{108}\approx 0.009 .
\ee

\begin{remark}

It is obvious that short intervals: $[T-G,T+G]$ with $G\in (0,1)$ are not included in the methods of Balasubramanian, Heath-Brown and Ivic leading to
(\ref{1.5}) and (\ref{1.6}), (\ref{1.7}).

\end{remark}

In this work we present a new method how to deal with short parts ($1\leq U\leq T^{1/3+\epsilon}$) and microscopic parts
($0<U<1$) of the Hardy-Littlewood integral
\be \label{1.8}
\int_T^{T+U}Z^2(t){\rm d}t .
\ee
To attain this goal we will use only elementary geometric properties of Jacob's ladders.

\begin{remark}

Our new and elementary method leads to new expressions for the integrals of type (\ref{1.8}) and these new expressions cannot be derived by means of
methods of Balasubramanian, Heath-Brown and Ivic.

\end{remark}

\begin{remark}

We will see that the Jacob's ladders (by means of the formula (\ref{1.3})) discover the feature of short and microscopic parts of the Hardy-Littlewood
integral. From this point of view, the Jabob's ladders play the role of \emph{Golem's shem} for the mentioned Hardy-Littlewood integral.

\end{remark}

\section{Tangent law}

\subsection{}

The basic idea is in the following theorem.

\begin{theorem}
For $0<U<U_0=T^{1/3+2\epsilon}$ the following is true
\begin{eqnarray} \label{2.1}
\int_T^{T+U}Z^2(t){\rm d}t&=&\left(U\ln\frac{\varphi(T)}{2}-aU\right)\tan\left[ \alpha(T,U)\right]+\nonumber \\
&+&\mcal{O}\left(\frac{1}{T^{1/3-4\epsilon}}\right),\ a=\ln 2\pi-1-c,
\end{eqnarray}
where $\alpha=\alpha(T,U)$ is the angle of the chord of the curve $y=1/2\varphi(T)$ that binds the points
\be \label{2.2}
[T,\frac{1}{2}\varphi(T)],\ [T+U,\frac{1}{2}\varphi(T+U)] .
\ee

\end{theorem}

\begin{proof}

As a consequence of (\ref{1.3}) we have
\begin{eqnarray} \label{2.3}
\int_T^{T+U}Z^2(t){\rm d}t&=&\frac{1}{2}[\varphi(T+U)-\varphi(T)]\left( \ln\frac{\varphi(T)}{2}-a\right)+\nonumber \\
&+&\mcal{O}\left\{ \frac{[\varphi(T+U)-\varphi(T)]^2}{T} \right\}+\mcal{O}\left(\frac{\ln T}{T}\right) .
\end{eqnarray}
Using the equality (see \cite{7}, (5.2))
\bdis
\frac{\varphi(T+U_0)-\varphi(T)}{T}=\mcal{O}(1) ,
\edis
and (\ref{2.3}) with $U=U_0$ and the Balasubramanian formula (\ref{1.5}) we obtain the following inequality
\be \label{2.4}
\varphi(T+U)-\varphi(T)\leq \varphi(T+U_0)-\varphi(T)=\mcal{O}(T^{1/3+2\epsilon}),\quad 0<U\leq U_0 .
\ee
And now by making use of formulae (\ref{2.3}) and (\ref{2.4}) we have
\be \label{2.5}
\int_T^{T+U}Z^2(t){\rm d}t=\frac{1}{2}[\varphi(T+U)-\varphi(T)]\ln\left( e^{-a}\frac{\varphi(T)}{2}\right)+\mcal{O}\left(\frac{1}{T^{1/3-4\epsilon}}\right) .
\ee
At this moment we need only to remind that
\bdis
\frac{1}{2}[\varphi(T+U)-\varphi(T)]=U\tan \alpha ,
\edis
where $\alpha=\alpha(T,U)$ is the angle of the chord binding the points (\ref{2.2}), and therefore (\ref{2.1}) follows from (\ref{2.5}).

\end{proof}

\subsection{}

By a finer comparison of the formulae (\ref{1.5}) and (\ref{2.1}), with $U=U_0$, and by making use of
\bdis
\frac{1}{\ln T}\ln \frac{\varphi(T)}{2}=1+\mcal{O}\left(\frac{1}{\ln^2 T}\right) ,
\edis
(for this formula see \cite{7}, (6.6)) we can derive the following asymptotic formula
\be \label{2.6}
\tan[\alpha(T,U_0)]=1-\frac{1-c}{\ln T} +\mcal{O}\left(\frac{1}{\ln^2 T}\right) ,
\ee
where $\alpha(T,U_0)$ is the angle of the chord binding the points
\be \label{2.7}
[T,\frac{1}{2}\varphi(T)],\ [T+U_0,\frac{1}{2}\varphi(T+U_0)] .
\ee

\begin{remark}

The asymptotic formula (\ref{2.6}) is the geometric consequence of the Balasubramanian formula (\ref{1.5}). The geometric consequence of the Ivic
estimate (\ref{1.7}) is the following estimate
\bdis
\tan[\alpha(T,G)]=\mcal{O}(\ln T) .
\edis

\end{remark}

\subsection{}

Let us consider the set of all chords of the curve $y=\frac{1}{2}\varphi(T)$ which are parallel to the main chord that binds the points (\ref{2.7}).
Let the generic chord of this class bind the points $[N,\frac{1}{2}\varphi(N)]$ and $[M,\frac{1}{2}\varphi(M)]$. Then from (\ref{2.1}) and
(\ref{2.6}) we obtain the following

\begin{cor}

There is continuum of intervals $[N,M]\subset [T,T+T^{1/3+2\epsilon}]$ for which the following formula holds true
\begin{eqnarray} \label{2.8}
\int_N^M Z^2(t){\rm d}t&=&(M-N)\ln N +(2c-\ln 2\pi)(M-N)+\nonumber \\
&+&\mcal{O}\left(\frac{M-N}{\ln T}\right)+\mcal{O}\left(\frac{1}{T^{1/3-4\epsilon}}\right) .
\end{eqnarray}

\end{cor}

\begin{remark}

Let us remark that the formula (\ref{2.8}) cannot be obtained by methods of Balasubramanian, and estimates done by Heath-Brown and Ivic,
since it holds true, for example, for system intervals of type: $[N,M]\subset (0,T^{1/6})$.

\end{remark}

\section{On microscopic parts of the Hardy-Littlewood integral in neighbourhoods of zeroes of the function $\zeta(\frac{1}{2}+iT)$}

Let $\gamma,\gamma'$ be a pair of neighbouring zeroes of the function $\zeta(\frac{1}{2}+iT)$. The function $\frac{1}{2}\varphi(T)$ is
necessarily convex on some right neighbourhood of the point $T=\gamma$, and this function is necessarily concave on some left
neighbourhood of the point $T=\gamma'$. Therefore, there exists the minimal value of $\rho\in (\gamma,\gamma')$ such that
$[\rho,\frac{1}{2}\varphi(\rho)]$ is the inflection point of the curve $y=\frac{1}{2}\varphi(T)$. At this point, by properties of the Jacob's ladders,
we have $\varphi^\prime(\rho)>0$. Let furthermore $\beta=\beta(\gamma,\rho)$ be the angle of the chord binding the points
\be \label{3.1}
[\gamma,\frac{1}{2}\varphi(\gamma)],\ [\rho,\frac{1}{2}\varphi(\rho)] .
\ee
Then we obtain by our Theorem.

\begin{cor}

For every sufficiently big zero $T=\gamma$ of the function $\zeta(\frac{1}{2}+iT)$ the following formulae describing microscopic parts of the
Hardy-Littlewood integral hold true
\begin{itemize}

\item[(A)] the continuum of formulae
\begin{eqnarray} \label{3.2}
\int_\gamma^{\gamma+U}Z^2(t){\rm d}t&=&\left( U\ln \frac{\varphi(\gamma)}{2}-aU\right)\tan\alpha+\mcal{O}\left( \frac{1}{\gamma^{1/3-4\epsilon}}\right),\nonumber \\
& & \alpha\in (0,\beta(\gamma,\rho)),\ U=U(\gamma,\alpha)\in (0,\rho-\gamma) ,
\end{eqnarray}
where $\alpha=\alpha(\gamma,U)$ is the angle of the rotating chord binding the points $[\gamma,\frac{1}{2}\varphi(\gamma)]$ and
$[\gamma+U,\frac{1}{2}\varphi(\gamma+U)]$,

\item[(B)] the continuum of formulae for chord parallel to the chord given by points (\ref{3.1})
\begin{eqnarray} \label{3.3}
\int_N^MZ^2(t){\rm d}t&=&\left[(M-N)\ln \frac{\varphi(N)}{2}-a(M-N) \right]\tan [\beta(\gamma,\rho)]+\nonumber \\
&+&\mcal{O}\left(\frac{1}{\gamma^{1/3-4\epsilon}}\right),\quad \gamma<N<M<\rho .
\end{eqnarray}
\end{itemize}

\end{cor}

\begin{remark}

The notion \emph{microscopic parts} of the Hardy-Littlewood integral has its natural origin in the following: by Karacuba's selbergian estimate
(see \cite{5}, p. 265) we know that for \emph{almost all} intervals $[\gamma,\gamma']\subset [T,T+T^{1/3+2\epsilon}]$ we have
\be \label{3.4}
\gamma'-\gamma<A\frac{\ln\ln T}{\ln T}\to 0,\ T\to\infty .
\ee

\end{remark}

\begin{remark}

In connection with (\ref{3.4}) we can remind that if the Riemann conjecture holds true then the Littlewood's estimate takes place
\bdis
\gamma'-\gamma<\frac{A}{\ln\ln \gamma}\to 0,\ \gamma\to\infty ,
\edis
(see \cite{8}, p. 296, simple consequence of the estimate $S(T)=\mcal{O}(\ln T/\ln\ln T)$).

\end{remark}

\begin{remark}

It is obvious that the formulae (\ref{3.2}) and (\ref{3.3}) for microscopic parts of the Hardy-Littlewood integral cannot be obtained from
the Balasubramanian's formula or by estimates by Heath-Brown and Ivic.

\end{remark}

\section{Second class of formulae for parts of the Hardy-Littlewood integral beginning in zeroes of the function $\zeta(\frac{1}{2}+iT)$}

Let $T=\gamma,\bar{\gamma}$ be a pair of zeroes of the function $\zeta(\frac{1}{2}+iT)$, where $\bar{\gamma}$ obeys the following conditions
\bdis
\bar{\gamma}=\gamma+\gamma^{1/3+2\epsilon}+\Delta(\gamma),\ 0\leq \Delta(\gamma)=\mcal{O}(\gamma^{1/4+\epsilon}) ,
\edis
(see the Hardy-Littlewood estimate for the distance between the neighbouring zeroes, \cite{3}, pp. 125, 177-184). Consequently,
\be \label{4.1}
U(\gamma)=\gamma^{1/3+2\epsilon}+\Delta(\gamma)\sim \gamma^{1/3+2\epsilon},\ \gamma\to\infty .
\ee
For the chord that binds the points
\be \label{4.2}
[\gamma,\frac{1}{2}\varphi(\gamma)],\ [\bar{\gamma},\frac{1}{2}\varphi(\bar{\gamma})]
\ee
we have by (\ref{2.6}) and (\ref{4.1})
\be \label{4.3}
\tan[\alpha(\gamma,U(\gamma))]=1-\frac{1-c}{\ln \gamma}+\mcal{O}\left(\frac{1}{\ln^2\gamma}\right) .
\ee
The continuous curve $y=\frac{1}{2}\varphi(T)$ lies below the chord given by points (\ref{4.2}) on some right neighbourhood of the point $T=\gamma$ and
this curve lies above that chord on some left neighbourhood of the point $T=\bar{\gamma}$. Therefore there exists a common point of the curve and
the chord. Let $[\bar{\rho},\frac{1}{2}\varphi(\bar{\rho})],\ \bar{\rho}\in (\gamma,\bar{\gamma})$ be such a common point that is the closest to
the point $[\gamma,\frac{1}{2}\varphi(\gamma)]$. Then we obtain from our Theorem

\begin{cor}

For every sufficiently big zero $T=\gamma$ of the function $\zeta(\frac{1}{2}+iT)$ we have the following formulae for the parts (\ref{1.8}) of the
Hardy-Littlewood integral
\begin{itemize}

\item[(A)] continuum of formulae for the rotating chord
\begin{eqnarray} \label{4.4}
\int_\gamma^{\gamma+U}Z^2(t){\rm d}t&=&\left( U\ln\frac{\varphi(\gamma)}{2}-aU\right)\tan\alpha +
\mcal{O}\left( \frac{1}{\gamma^{1/3-4\epsilon}}\right), \nonumber \\
& & \tan\alpha\in [\eta,1-\eta],\ U=U(\gamma,\alpha)\in (0,\bar{\rho}-\gamma) ,
\end{eqnarray}
where $\alpha=\alpha(\gamma,U)$ is the angle of the rotating chord binding the points $[\gamma,\frac{1}{2}\varphi(\gamma)]$ and
$[\gamma+U,\frac{1}{2}\varphi(\gamma+U)]$, and $0<\eta$ is an arbitrarily small number,

\item[(B)] continuum of formulae for the chords parallel to the chord binding the points (\ref{4.2})
\begin{eqnarray} \label{4.5}
\int_N^MZ^2(t){\rm d}t&=&(M-N)\ln N+(2c-\ln 2\pi)(M-N)+\nonumber \\
&+&\mcal{O}\left(\frac{M-N}{\ln \gamma}\right)+\mcal{O}\left(\frac{1}{\gamma^{1/3-4\epsilon}}\right),\ \gamma\leq N<M\leq \bar{\rho} .
\end{eqnarray}

\end{itemize}

\end{cor}

\begin{remark}

For example, in the case $\alpha=\pi/6$ we have from (\ref{4.4})
\bdis
\int_\gamma^{\gamma+U(\pi/6)}Z^2(t){\rm d}t=\frac{1}{\sqrt{3}}(U\ln\gamma-aU)+\mcal{O}\left(\frac{U}{\ln\gamma}\right)+
\mcal{O}\left(\frac{1}{\gamma^{1/3-4\epsilon}}\right),
\edis
for every sufficiently big zero $T=\gamma$ of the function $\zeta(\frac{1}{2}+iT)$.

\end{remark}

\section{On I.M. Vinogradov's scepticism on possibilities of the method of trigonometric sums}

\subsection{}

I.M. Vinogradov, in the Introduction to his monograph \cite{9},
has analyzed the possibilities of the method of trigonometric sums (H. Weil' sums) in the problem of estimation of the remainder
term $R(N)$ in the asymptotic formula
\be \label{5.1}
\pi(N)-\int_2^N\frac{{\rm d}x}{\ln(x)}=R(N)
\ee
(see \cite{9}, p. 13). He made the following remark in this: \\
\emph{Obviously, it is very hard to move on essentially in solution of the problem to find the order of the $R$ term (willing to find
$R=\mcal{O}(N^{1-c}),\ c=0.000001$) by making use of only some improvements of the H. Weyl's estimates and without making use of further important
progresses in the theory of the zeta-function.} \\

\subsection{}

We will discuss in this section an analogue of the Vinogradov's scepticism in our case of estimation of the remainder term for the Hardy-Littlewood
integral
\be \label{5.2}
\int_0^TZ^2(t){\rm d}t-T\ln T-(2c-1-\ln 2\pi)T=R(T) ,
\ee
(analogue to (\ref{5.1}). The first mathematician who applied the method of trigonometric sums to estimation of $R(T)$ was Titchmarsh (in 1934),
and he received the result $R(T)=\mcal{O}(T^{5/12+\epsilon})$, (see \cite{8}, p. 123). Balasubramanian then improved the
Titchmarsh' result to $R(T)=\mcal{O}(T^{1/3+\epsilon})$, \cite{1}. The crucial result in this field obtained Good \cite{2}:
$R(T)=\Omega(T^{1/4})$.

\begin{remark}

It follows the Good's result that the estimate of type $R(T)=\mcal{O}(T^{1/4+\epsilon})$ - which still represent an unbounded and
unremovable absolute error in the formula for the Hardy-Littlewood integral (see Remark 1) - is the final result for the method of
trigonometric sums in this question.

\end{remark}

\begin{remark}

Since our almost exact formula (\ref{1.3}) has been proved independently on the method of trigonometric sums, it, together with the
Balasubramanian formula (\ref{1.4}), can give a good keystone for an analogue of the I.M. Vinogradov's scepticism in the problem of possibility
to find a finer representation of the Hardy-Littlewood integral $\int_0^T Z^2(t){\rm d}t$.

\end{remark}

\subsection{}

I would like to describe my two main goals when I was working in Titchmarsh' sequences $\{ Z(t_\nu) \}$, where $\{ t_\nu\}$ is the Gramm's
sequence. I wanted to
\begin{itemize}

\item[(a)] improve the knowledge about the local variant of the classical Titchmarsh formulae (see \cite{3}, pp. 221, 222)
\begin{eqnarray*}
& &
\sum_{\nu=\nu_0}^NZ(t_{2\nu})=2N+\mcal{O}(N^{3/4}\ln^{3/4}N), \\
& &
\sum_{\nu=\nu_0}^NZ(t_{2\nu+1})=-2N+\mcal{O}(N^{3/4}\ln^{3/4}N) ,
\end{eqnarray*}

\item[(b)] prove mean-value theorems for the function $Z(t)$ on related non-connected sets.

\end{itemize}

To solve the tasks (a) and (b) I have first used the method of trigonometric sums (see \cite{5}, p. 260, 265, 266, \cite{4}, p. 37).
In the task (a) I have improved the Titchmarsh exponent as $3/4\to 1/6$, i.e. I improved that exponent by 77.7\%. \\

In the task (b) I obtained a new class of mean-value theorems (see \cite{6}), corresponding to the exponent $1/6$

\begin{eqnarray*}
& &
\frac{1}{m\{ G_1(x)\}}\int_{G_1(x)}Z(t){\rm d}t\sim 2\frac{\sin(x)}{x} , \\
& &
\frac{1}{m\{ G_2(y)\}}\int_{G_2(y)}Z(t){\rm d}t\sim -2\frac{\sin(y)}{y},\quad T\to\infty .
\end{eqnarray*}

In connection with the task (a) I have also essentially improved the Hardy-Littlewood exponent $1/4$ (since 1918) to $1/6$ (problem of
estimation of distance of neighbouring zeroes of the function $\zeta(\frac{1}{2}+iT)$, see \cite{3}, p. 125, 177-184). \\

Let us follow the sequence of improvements:
\begin{itemize}
\item Moser - 33.3\% improvement of the Hardy-Littlewood exponent $1/4$
\item Karacuba - 6.25\% improvement of the exponent $1/6$
\item Ivic - 0.19\% improvement of the Karacuba's exponent.

\end{itemize}

\begin{remark}

The sequence of improvements of the exponent of type $0.19\%, \dots $ show that scepticism of I.M. Vinogradov takes place also in
the possibility of successful application of the method of trigonometric sums in the problem of crucial improvement of the exponent $1/6$.

\end{remark}

\thanks{I would like to thank Michal Demetrian for helping me with the electronic form of this work.}

\end{document}